\documentclass[12pt]{amsart}

\usepackage{amsmath, amsthm, amssymb, latexsym, enumerate, stmaryrd}
\usepackage[all]{xy}
\usepackage[mathscr]{eucal}
\DeclareMathOperator{\gr}{graph}

\renewcommand{\epsilon}{\varepsilon}

\newcommand{\sub}[1]{_{\textup{\tiny{\fontfamily{cmr}\selectfont #1}}}}

\newcommand{\uhr}{\upharpoonright}

\renewcommand{\phi}{\varphi}
\renewcommand{\geq}{\geqslant}
\renewcommand{\leq}{\leqslant}

\renewcommand{\nleq}{\nleqslant}

\newcommand{\open}[1]{\llbracket#1\rrbracket}

\newcommand{\converges}{\mathord{\downarrow}}
\newcommand{\diverges}{\mathord{\uparrow}}

\newtheorem{thm}{Theorem}[section]
\newtheorem{prop}[thm]{Proposition}
\newtheorem{lem}[thm]{Lemma}
\newtheorem{cor}[thm]{Corollary}

\theoremstyle{definition}

\newtheorem{defn}[thm]{Definition}
\newtheorem{oq}[thm]{Open Question}

\begin{document}

\title[Coarse Reducibility and Algorithmic Randomness]{Coarse
Reducibility and \\ Algorithmic Randomness}

\date{\today}

\author[D. R. Hirschfeldt]{Denis R. Hirschfeldt}

\address{Department of Mathematics\\ University of Chicago}

\email{drh@math.uchicago.edu}

\thanks{Hirschfeldt was partially supported by grant
DMS-1101458 from the National Science Foundation of the United
States.}

\author[C. G. Jockusch, Jr.]{Carl G. Jockusch, Jr.}

\address{Department of Mathematics\\University of Illinois at
Urbana-Cham\-paign}

\email{jockusch@math.uiuc.edu}

\author[R. Kuyper]{Rutger Kuyper}

\thanks{Kuyper's research was supported by NWO/DIAMANT grant
613.009.011 and by John Templeton Foundation grant 15619: ``Mind,
Mechanism and Mathematics: Turing Centenary Research Project''.}

\address{Department of Mathematics\\Radboud University Nijmegen}

\email{r.kuyper@math.ru.nl}

\author[P. E. Schupp]{Paul E. Schupp}

\address{Department of Mathematics\\University of Illinois at
Urbana-Cham\-paign}

\email{schupp@illinois.edu}



\begin{abstract}
\vspace{1.2cm} 
A \emph{coarse description} of a set $A \subseteq \omega$ is a set $D
\subseteq \omega$ such that the symmetric difference of $A$ and $D$
has asymptotic density $0$. We study the extent to which noncomputable
information can be effectively recovered from all coarse descriptions
of a given set $A$, especially when $A$ is effectively random in some
sense. We show that if $A$ is $1$-random and $B$ is computable from
every coarse description $D$ of $A$, then $B$ is $K$-trivial, which
implies that if $A$ is in fact weakly $2$-random then $B$ is
computable. Our main tool is a kind of compactness theorem for
cone-avoiding descriptions, which also allows us to prove the same
result for $1$-genericity in place of weak $2$-randomness. In the
other direction, we show that if $A \leq\sub{T} \emptyset'$ is a
$1$-random set, then there is a noncomputable c.e.\ set computable
from every coarse description of $A$, but that not all $K$-trivial
sets are computable from every coarse description of some $1$-random
set.  We study both uniform and nonuniform  notions of coarse reducibility. 
A set  $Y$ is \emph{uniformly coarsely reducible} to $X$
 if there is a Turing functional $\Phi$ such that if
$D$ is a coarse description of $X$, then $\Phi^D$ is a coarse
description of $Y$.  A set  $B$ is \emph{nonuniformly coarsely reducible} to $A$
if every coarse description of $A$ computes a coarse description of
$B$.   We  show that a certain natural embedding of the Turing degrees into the
  coarse degrees (both uniform and nonuniform) is not surjective. 
We also show that if
two sets are mutually weakly $3$-random, then their coarse
degrees form a minimal pair, in both the uniform and nonuniform cases,
but that the same is not true of every
pair of relatively $2$-random sets, at least in the nonuniform coarse
degrees.
\end{abstract}

\maketitle

\section{Introduction}

There are many natural problems with high worst-case complexity that
are nevertheless easy to solve in most instances. The notion of
``generic-case complexity''  was introduced by Kapovich, Myasnikov, Schupp,
and Shpilrain \cite{KMSS} as a notion that is more  tractable than  average-case
complexity but  still allows a  somewhat nuanced analysis of such problems.
That paper also introduced the idea of  generic computability, which 
captures the idea of having a partial algorithm that  correctly
computes $A(n)$ for ``almost all'' $n$, while never giving an
incorrect answer. Jockusch and Schupp \cite{JS} began the general computability
theoretic investigation of generic computability and also defined the idea of
coarse computability,  which   captures the idea
of having a total algorithm that always answers and  may make mistakes, but correctly computes $A(n)$ for
``almost all'' $n$.  We are here concerned with this latter
concept. We first need a good notion of ``almost all''
natural numbers.

\begin{defn}
Let $A \subseteq \omega$. The \emph{density of $A$ below $n$}, denoted by
$\rho_n(A)$, is $\frac{|A \uhr n|}{n}$. The \emph{upper density}
$\overline{\rho}(A)$ of $A$ is $\limsup_n \rho_n(A)$. The
\emph{lower density} $\underline{\rho}(A)$ of $A$ is $\liminf_n
\rho_n(A)$. If $\overline{\rho}(A)=\underline{\rho}(A)$ then we call
this quantity the \emph{density} of $A$, and denote it by $\rho(A)$.

We say that $D$ is a \emph{coarse description} of $X$ if $\rho(D
\triangle X)=0$, where $\triangle$ denotes symmetric difference. A set
$X$ is \emph{coarsely computable} if it has a computable coarse
description.
\end{defn}

This idea leads to  natural notions of reducibility.

\begin{defn}
We say that $Y$ is \emph{uniformly coarsely reducible} to $X$, and write $Y
\leq\sub{uc} X$, if there is a Turing functional $\Phi$ such that if
$D$ is a coarse description of $X$, then $\Phi^D$ is a coarse
description of $Y$. This reducibility induces an equivalence relation
$\equiv\sub{uc}$ on $2^\omega$. We call the equivalence class of $X$
under this relation the \emph{uniform coarse degree} of $X$.
\end{defn}

Uniform coarse reducibility, generic reducibility (defined in \cite{JS}), and
several related reducibilities have been termed \emph{notions of
robust information coding} by Dzhafarov and Igusa \cite{DI}. Work on
such notions has mainly  focused on their uniform versions. (One
exception is a  result on nonuniform ii-reducibility in Hirschfeldt
and Jockusch \cite{HJ}.)  However, their nonuniform versions also seem
to be of interest. In particular, we will  work with  the following
nonuniform version of coarse reducibility.

\begin{defn}
We say that $Y$ is \emph{nonuniformly coarsely reducible} to $X$, and
write $Y \leq\sub{nc} X$, if every coarse description of $X$ computes
a coarse description of $Y$. This reducibility induces an equivalence
relation $\equiv\sub{nc}$ on $2^\omega$. We call the equivalence class
of $X$ under this relation the \emph{nonuniform coarse degree} of $X$.
\end{defn}

Note that the coarsely computable sets form the least degree in both
the uniform and nonuniform coarse degrees. Uniform coarse reducibility
clearly implies nonuniform coarse reducibility. We will show in the
next section that, as one might expect, the converse fails.  The
development of the theory of notions of robust information coding and
related concepts have led to interactions with computability theory
(as in Jockusch and Schupp \cite{JS}; Downey, Jockusch, and Schupp
\cite{DJS}; Downey, Jockusch, McNicholl, and Schupp \cite{DJMS}; and
Hirschfeldt, Jockusch, McNicholl, and Schupp \cite{HJMS}), reverse
mathematics (as in Dzhafarov and Igusa \cite{DI} and Hirschfeldt and
Jockusch \cite{HJ}), and algorithmic randomness (as in Astor
\cite{A}).

In this paper, we investigate connections between coarse reducibility
and algorithmic randomness. In Section \ref{embed}, we describe
natural embeddings of the Turing degrees into the uniform and 
nonuniform coarse degrees, and discuss some of their basic
properties. In Section \ref{Ktriv}, we show that no weakly $2$-random
set can be in the images of these embeddings by showing that if $X$ is
weakly $2$-random and $A$ is noncomputable, then there is some coarse
description of $X$ that does not compute $A$. More generally, we show
that if $X$ is $1$-random and $A$ is computable from every coarse
description of $X$, then $A$ is $K$-trivial. Our main tool is a kind
of compactness theorem for cone-avoiding descriptions. We also show
that there do exist noncomputable sets computable from every coarse
description of some $1$-random set, but that not all $K$-trivial sets
have this property. In Section \ref{further}, we give further examples
of classes of sets that cannot be in the images of our embeddings. In
Section \ref{minpairs}, we show that if two sets are relatively weakly
$3$-random then their coarse degrees form a minimal pair, in both the
uniform and nonuniform cases,
but that, at least for the nonuniform coarse degrees, the same is not
true of every pair of relatively $2$-random sets. These results are
analogous to the fact that, for the Turing degrees, two relatively
weakly $2$-random sets always form a minimal pair, but two relatively
$1$-random sets may not. In Section \ref{questions}, we conclude with
a few open questions.

We assume familiarity with basic notions of computability theory (as
in \cite{So}) and algorithmic randomness (as in \cite{DH} or
\cite{N}). For $S \subseteq 2^{<\omega}$, we write $\open{S}$ for the
open subset of $2^\omega$ generated by $S$; that is, 
$\open{S} = \{X : \exists n \,(X \uhr n \in S)\}$. 
We denote the uniform measure on $2^\omega$ by $\mu$.

\section{Coarsenings and embeddings of the Turing degrees}
\label{embed}

We can embed the Turing degrees into both the uniform  and
nonuniform coarse degrees, and our first connection between coarse
computability and algorithmic randomness  comes  from considering such
embeddings. While there may be several ways to define such embeddings,
a natural way to proceed is to define a map $\mathcal{C} : 2^\omega \rightarrow
2^\omega$ such that $\mathcal{C}(A)$ contains the same information as $A$, but
coded in a ``coarsely robust'' way. That is, we would like $\mathcal{C}(A)$ to
be computable from $A$, and $A$ to be computable from any coarse
description of $\mathcal{C}(A)$.

In the case of the uniform coarse degrees, one might think that the
latter reduction should be uniform, but that condition would be too
strong: If $\Gamma^D=A$ for every coarse description $D$ of $\mathcal
C(A)$ then $\Gamma^\sigma(n)\converges \; \Rightarrow \;
\Gamma^\sigma(n)=A(n)$ (since every string can be extended to a coarse
description of $\mathcal C(A)$), which, together with the fact that for
each $n$ there is a $\sigma$ such that $\Gamma^\sigma(n)\converges$,
implies that $A$ is computable. Thus we relax the uniformity condition
slightly in the following definition.

\begin{defn}
\label{coarseningdefn}
A map $\mathcal{C} : 2^\omega \rightarrow 2^\omega$ is a \emph{coarsening} if
for each $A$ we have $\mathcal{C}(A) \leq\sub{T} A$, and for each coarse
description $D$ of $\mathcal{C}(A)$, we have $A \leq\sub{T} D$. A coarsening $\mathcal{C}$
is \emph{uniform} if there is a binary Turing functional $\Gamma$ with
the following properties for every coarse description $D$ of
$\mathcal{C}(A)$:
\begin{enumerate}[\rm 1.]

\item  $\Gamma^D$ is total. 

\item  Let $A_s(n)=\Gamma^D(n,s)$. Then $A_s=A$ for cofinitely many $s$.
\end{enumerate}

\end{defn}

\begin{prop}\label{coarseningsprop}
Let $\mathcal{C}$ and $\mathcal{F}$ be coarsenings and $A$ and $B$ be sets. Then
\begin{enumerate}[\rm 1.]

\item $B \leq\sub{T} A$ if and only if $\mathcal{C}(B) \leq\sub{nc} \mathcal{C}(A)$.

\item If  $\mathcal{C}$ is uniform then $B \leq\sub{T} A$ if and only if $\mathcal{C}(B) \leq\sub{uc}
\mathcal{C}(A)$.

\item $\mathcal{C}(A) \equiv\sub{nc} \mathcal{F}(A)$, and

\item if $\mathcal{C}$ and $\mathcal{F}$ are both uniform then $\mathcal{C}(A) \equiv\sub{uc} \mathcal{F}(A)$.

\end{enumerate}
\end{prop}

\begin{proof}
1. Suppose that $\mathcal{C}(B) \leq\sub{nc} \mathcal{C}(A)$. Then $\mathcal{C}(A)$ computes a
coarse description $D_1$ of $\mathcal{C}(B)$. Thus $B \leq\sub{T} D_1 \leq\sub{T}
\mathcal{C}(A) \leq\sub{T} A$.

Now suppose that $B \leq\sub{T} A$ and let $D_2$ be a coarse description
of $\mathcal{C}(A)$. Then $\mathcal{C}(B) \leq\sub{T} B \leq\sub{T} A \leq\sub{T} D_2$. Thus
$\mathcal{C}(B) \leq\sub{nc} \mathcal{C}(A)$.

2. Suppose that $\mathcal{C}$ is uniform and that $B \leq\sub{T} A$.  Let
$D_2$ be a coarse description of $\mathcal{C}(A)$.  Let $A_s$ be as in
Definition \ref{coarseningdefn}, with $D=D_2$.  Then $\mathcal{C}(B) \leq\sub{T}
B \leq\sub{T} A$, so let $\Phi$ be such that $\Phi^A=\mathcal{C}(B)$.  Let $X
\leq\sub{T} D_2$ be defined as follows. Given $n$, search for an $s>n$
such that $\Phi^{A_s}(n)\converges$ and let
$X(n)=\Phi^{A_s}(n)$. (Note that such an $s$ must exist.) Then
$X(n)=\Phi^A(n)=\mathcal{C}(B)(n)$ for almost all $n$, so $X$ is a coarse
description of $\mathcal{C}(B)$. Since $X$ is obtained uniformly from $D_2$, we
have $\mathcal{C}(B) \leq\sub{uc} \mathcal{C}(A)$.  The converse follows immediately from
1.

3. Let $D_3$ be a coarse description of $\mathcal{F}(A)$. Then $\mathcal{C}(A) \leq\sub{T}
A \leq\sub{T} D_3$. Thus $\mathcal{C}(A) \leq\sub{nc} \mathcal{F}(A)$. By symmetry, $\mathcal{C}(A)
\equiv\sub{nc} \mathcal{F}(A)$.

4. If $\mathcal{F}$ is uniform then the same argument as in the proof of 2 shows
that we can obtain a coarse description of $\mathcal{C}(A)$ uniformly from $D_3$,
whence $\mathcal{C}(A) \leq\sub{uc} \mathcal{F}(A)$. If $\mathcal{C}$ is also uniform then
 $\mathcal{C}(A) \equiv\sub{uc} \mathcal{F}(A) $  by symmetry.
\end{proof}

Thus uniform coarsenings all induce the same natural embeddings. It
remains to show that uniform coarsenings exist.
One  example  is given by Dzhafarov and Igusa \cite{DI}.
 We give a similar  example.  Let
$I_n=[n!,(n+1)!)$ and let $\mathcal{I}(A)=\bigcup_{n \in A} I_n$; this map first appeared in Jockusch and Schupp \cite{JS}. Clearly $\mathcal{I}(A)
\leq\sub{T} A$, and it is easy to check that if $D$ is a coarse
description of $\mathcal{I}(A)$ then $D$ computes $A$. Thus $\mathcal{I}$ is a
coarsening.

To construct a uniform coarsening, let
$\mathcal{H}(A)=\{\langle n,i \rangle : n \in A\; \wedge\; i \in \omega\}$ and
define $\mathcal{E}(A)=\mathcal{I}(\mathcal{H}(A))$.  The notation $\mathcal{E}$  denotes 
this particular coarsening throughout the paper.

\begin{prop}\label{Euniform}
The map $\mathcal{E}$ is a uniform coarsening.
\end{prop}

\begin{proof}
Clearly $\mathcal{E}(A) \leq\sub{T} A$. Now let $D$ be a coarse description of
$\mathcal{E}(A)$. Let $G =\{m : |D \cap I_m|>\frac{|I_m|}{2}\}$ and let $A_s=\{n
: \langle n,s \rangle \in G\}$. Then $G =^*\mathcal{H}(A)$, so $A_s=A$ for all
but finitely many $s$, and the $A_s$ are obtained uniformly from $D$.
\end{proof}

A first natural question is whether uniform coarse reducibility and non\-uniform
coarse reducibility are indeed different. We give a positive answer by
showing that, unlike in the nonuniform case, the mappings $\mathcal{E}$ and $\mathcal{I}$ are not
equivalent up to uniform coarse reducibility. Recall that a set $X$ is
\emph{autoreducible} if there exists a Turing functional $\Phi$ such
that for every $n \in \omega$ we have $\Phi^{X \setminus \{n\}}(n) =
X(n)$. Equivalently, we could require that $\Phi$ not ask whether its
input belongs to its oracle.  We now introduce a $\Delta^0_2$-version
of this notion.

\begin{defn}
A set $X$ is \emph{jump-autoreducible} if there exists a Turing
functional $\Phi$ such that for every $n \in \omega$ we have $\Phi^{(X
\setminus \{n\})'}(n) = X(n)$.
\end{defn}

\begin{prop}
\label{cjumpauto}
Let $X$ be such that $\mathcal{E}(X) \leq\sub{uc} \mathcal{I}(X)$. Then $X$ is
jump-autoreducible.
\end{prop}

\begin{proof}
We must give a procedure for computing $X(n)$ from $(X \setminus
\{n\})'$ that is uniform in $X$.  Given an oracle for $X \setminus
\{n\}$, we can uniformly compute $\mathcal{I}(X \setminus \{n\})$.  Now $\mathcal{I}(X
\setminus \{n\}) =^* \mathcal{I}(X)$, so $\mathcal{I}(X \setminus \{n\})$ is a coarse
description of $\mathcal{I}(X)$.  Since $\mathcal{E}(X) \leq\sub{uc} \mathcal{I}(X)$ by assumption, 
from $\mathcal{I}(X \setminus \{n\})$ we can uniformly compute a coarse description $D$ of
$\mathcal{E}(X)$. Since $\mathcal{E}$ is a uniform coarsening by Proposition
\ref{Euniform}, from $D$ we can uniformly obtain sets $A_0, A_1,
\dots$ with $A_s = X$ for all sufficiently large $s$.  Composing these
various reductions, from $X \setminus \{n\}$ we can uniformly compute
sets $A_0, A_1, \dots$ with $A_s = X$ for all sufficiently large $s$.
Then from $(X \setminus \{n\})'$ we can uniformly compute $\lim_s
A_s(n) = X(n)$, as needed.
\end{proof}

We will now show that $2$-generic sets are not jump-autoreducible, which
will give us a first example separating uniform coarse reducibility and
nonuniform coarse reducibility.  For this we first show that no
$1$-generic set is autoreducible, which is an easy exercise.

\begin{prop}
If $X$ is $1$-generic, then $X$ is not autoreducible.
\end{prop}

\begin{proof}
Suppose for the sake of a contradiction that $X$ is $1$-generic and is
autoreducible via $\Phi$. For a string $\sigma$, let $\sigma^{-1}(i)$
be the set of $n$ such that $\sigma(n)=i$.  If $\tau$ is a binary
string, let $\tau \setminus \{n\}$ be the unique binary string $\mu$
of the same length such that $\mu^{-1}(1) = \tau^{-1}(1) \setminus
\{n\}$. Let $S$ be the set of strings $\tau$ such that $\Phi^{\tau
\setminus \{n\}}(n) \converges \neq \tau(n) \converges$ for some $n$.
Then $S$ is a c.e.\ set of strings and $X$ does not meet $S$. Since
$X$ is $1$-generic, there is a string $\sigma \prec X$ that has no
extension in $S$. Let $n = |\sigma|$, and let $\tau \succ \sigma$ be a
string such that $\Phi^{\tau \setminus \{n\}}(n) \converges$.  Such a
string $\tau$ exists because $\sigma \prec X$ and $\Phi$ witnesses
that $X$ is autoreducible. Furthermore, we may assume that $\tau(n)
\neq \Phi^{\tau \setminus \{n\}}$, since changing the value of
$\tau(n)$ does not affect any of the conditions in the choice of
$\tau$.  Hence $\tau$ is an extension of $\sigma$ and $\tau \in S$,
which is the desired contradiction.
\end{proof}

\begin{prop}
\label{genjumpauto}
If $X$ is $2$-generic, then $X$ is not jump-autoreducible.
\end{prop}

\begin{proof}
Since $X$ is $2$-generic, $X$ is $1$-generic relative to $\emptyset'$.
Hence, by relativizing the proof of the previous proposition to
$\emptyset'$, we see that $X$ is not autoreducible relative to
$\emptyset'$.  However, the class of $1$-generic sets is uniformly
$\mathrm{GL}_1$, i.e., there exists a single Turing functional $\Psi$
such that for every $1$-generic $X$ we have $\Psi^{X \oplus
\emptyset'} = X'$, as can be verified by looking at the usual proof
that every $1$-generic is $\mathrm{GL}_1$ (see \cite[Lemma
2.6]{J}). Of course, if $X$ is $1$-generic, then $X \setminus \{n\}$
is also $1$-generic for every $n$.  Thus from an oracle for $(X
\setminus \{n\}) \oplus \emptyset'$ we can uniformly compute $(X
\setminus \{n\})'$. Now, if $X$ is jump-autoreducible, from $(X
\setminus \{n\})'$ we can uniformly compute $X(n)$.  Composing these
reductions shows that $X(n)$ is uniformly computable from $(X
\setminus \{n\}) \oplus \emptyset'$, which contradicts our previous
remark that $X$ is not autoreducible relative to $\emptyset'$.
\end{proof}

\begin{cor}
If $X$ is $2$-generic, then $\mathcal{E}(X) \leq\sub{nc} \mathcal{I}(X)$ but $\mathcal{E}(X)
\nleq\sub{uc} \mathcal{I}(X)$.
\end{cor}

\begin{proof}
We know that $\mathcal{E}(X) \leq\sub{nc} \mathcal{I}(X)$ from
Proposition \ref{coarseningsprop}. The fact that $\mathcal{E}(X) \nleq\sub{uc}
\mathcal{I}(X)$ follows from Propositions \ref{cjumpauto} and
\ref{genjumpauto}.
\end{proof}

It is natural to ask whether the same result holds for $2$-random
sets. In the proof above we used the fact that the $2$-generic sets
are uniformly $\mathrm{GL}_1$. For $2$-random sets this fact is almost
true, as expressed by the following lemma. The proof is adapted from
Monin \cite{M}, where a generalization for higher levels of randomness
is proved. Let $U_0,U_1,\ldots$ be a fixed universal Martin-L\"of test
relative to $\emptyset'$. The \emph{$2$-randomness deficiency} of a
$2$-random $X$ is the least $c$ such that $X \notin U_c$.

\begin{lem}
There is a Turing functional $\Theta$ such that, for a $2$-random $X$
and an upper bound $b$ on the $2$-randomness deficiency of $X$, we
have $\Theta^{X \oplus \emptyset',b} = X'$.
\end{lem}

\begin{proof}
Let $\mathscr{V}_e=\{Z : e \in Z'\}$. The $\mathscr{V}_e$ are uniformly $\Sigma^0_1$
classes, so we can define a function $f \leq\sub{T} \emptyset'$ such
that $\mu(\mathscr{V}_e \setminus \mathscr{V}_e[f(e,i)])<2^{-i}$ for all $e$ and $i$. Then
each sequence $\mathscr{V}_e \setminus \mathscr{V}_e[f(e,0)],\mathscr{V}_e \setminus
\mathscr{V}_e[f(e,1)],\ldots$ is an $\emptyset'$-Martin L\"of test, and from $b$
we can compute a number $m$ such that if $X$ is $2$-random and $b$
bounds the $2$-randomness deficiency of $X$, then $X \notin \mathscr{V}_e
\setminus \mathscr{V}_e[f(e,m)]$. Then $X \in \mathscr{V}_e$ if and only if $X \in \mathscr{V}_e[f(e,m)]$,
which we can verify $(X\oplus\emptyset')$-computably.
\end{proof}

\begin{prop}
If $X$ is $2$-random, then $X$ is not jump-autoreducible.
\end{prop}

\begin{proof}
Because $X$ is $2$-random, it is not autoreducible relative to
$\emptyset'$, as can be seen by relativizing the proof of Figueira,
Miller, and Nies \cite{FMN} that no $1$-random set is autoreducible.
To obtain a contradiction,  assume that $X$ is jump-autoreducible
through some functional $\Phi$. It can be directly verified that there
is a computable function $f$ such that $f(n)$ bounds the randomness
deficiency of $X \setminus \{n\}$. Now let $\Psi^{Y \oplus
\emptyset'}(n) = \Phi^{\Theta^{Y \oplus \emptyset',f(n)}}(n)$. Then
$X$ is autoreducible relative to $\emptyset'$ through $\Psi$, a
contradiction.
\end{proof}

\begin{cor}
If $X$ is $2$-random, then $\mathcal{E}(X) \leq\sub{nc} \mathcal{I}(X)$ but $\mathcal{E}(X)
\nleq\sub{uc} \mathcal{I}(X)$.
\end{cor}

Although we will not discuss generic reducibility after this section,
it is worth noting that our maps $\mathcal{E}$ and $\mathcal{I}$ also allow us to
distinguish generic reducibility from its nonuniform analog. Let us
briefly review the relevant definitions from \cite{JS}. A
\emph{generic description} of a set $A$ is a partial function that
agrees with $A$ where defined, and whose domain has density $1$. A set
$A$ is \emph{generically reducible} to a set $B$, written $A
\leq\sub{g} B$, if there is an enumeration operator $W$ such that if
$\Phi$ is a generic description of $B$, then $W^{\gr(\Phi)}$ is the
graph of a generic description of $A$. We can define the notion of
\emph{nonuniform generic reducibility} in a similar way: $A
\leq\sub{ng} B$ if for every generic description $\Phi$ of $B$, there
is a generic description $\Psi$ of $A$ such that $\gr(\Psi)$ is
enumeration reducible to $\gr(\Phi)$.

It is easy to see that $\mathcal{E}(X) \leq\sub{ng} \mathcal{I}(X)$ for all $X$. On the
other hand, we have the following fact.

\begin{prop}
If $\mathcal{E}(X) \leq\sub{g} \mathcal{I}(X)$ then $X$ is autoreducible.
\end{prop}

\begin{proof}
Let $I_n$ be as in the definition of $\mathcal{I}$. Suppose that $W$ witnesses
that $\mathcal{E}(X) \leq\sub{g} \mathcal{I}(X)$. We can assume that $W^Z$ is the graph of
a partial function for every oracle $Z$. Define a Turing functional
$\Theta$ as follows. Given an oracle $Y$ and an input $n$, let
$\Phi(k)=Y(m)$ if $k \in I_m$ and $m \neq n$, and let
$\Phi(k)\diverges$ if $k \in I_n$. Let $\Psi$ be the partial function
with graph $W^{\gr(\Phi)}$. Search for an $i$ and a $k \in I_{\langle
n,i \rangle}$ such that $\Psi(k)\converges$. If such numbers are
found then let $\Theta^Y(n)=\Psi(k)$. If $Y=X \setminus \{n\}$ then
$\Phi$ is a generic description of $\mathcal{I}(X)$, so $\Psi$ is a generic
description of $\mathcal{E}(X)$, and hence $\Theta^Y(n)\converges=X(n)$. Thus
$X$ is autoreducible.
\end{proof}

We finish this section by showing that, for both the uniform
and the nonuniform coarse degrees, coarsenings of the appropriate type preserve
joins but do not always preserve existing meets.

\begin{prop}
Let $\mathcal{C}$ be a coarsening. Then $\mathcal{C}(A \oplus B)$ is the least upper bound
of $\mathcal{C}(A)$ and $\mathcal{C}(B)$ in the nonuniform coarse degrees. The same holds
for the uniform coarse degrees if $\mathcal{C}$ is a uniform coarsening.
\end{prop}

\begin{proof}
By Proposition \ref{coarseningsprop} we know that $\mathcal{C}(A \oplus B)$ is
an upper bound for $\mathcal{C}(A)$ and $\mathcal{C}(B)$ in both the uniform  and
 nonuniform coarse degrees.  Let us show that it is the least upper
bound. If $\mathcal{C}(A),\mathcal{C}(B) \leq\sub{nc} G$ then every coarse description $D$
of $G$ computes both $A$ and $B$, so $D \geq\sub{T} A \oplus B
\geq\sub{T} \mathcal{C}(A \oplus B)$. Thus $G \geq\sub{nc} \mathcal{C}(A \oplus B)$.

Finally, assume that $\mathcal{C}$ is a uniform coarsening and let
$\mathcal{C}(A),\mathcal{C}(B) \leq\sub{uc} G$. Let $\Phi$ be a Turing
functional such that $\Phi^{A \oplus B}=\mathcal C(A \oplus B)$.
Every coarse description $H$ of $G$ uniformly
computes coarse descriptions $D_1$ of $\mathcal{C}(A)$ and $D_2$ of $\mathcal{C}(B)$.  
Since $\mathcal{C}$ is uniform, there are Turing functionals $\Gamma$ and $\Delta$
such that, letting $A_s(n)=\Gamma^{D_1}(n,s)$ and
$B_s(n)=\Gamma^{D_2}(n,s)$, we have that $A \oplus B = A_s \oplus B_s$
for all sufficiently large $s$. Let $E$ be
defined as follows. Given $n$, search for an $s \geq n$ such that
$\Phi^{A_s \oplus B_s}(n)\converges$, and let $E(n)=\Phi^{A_s \oplus
B_s}(n)$. If $n$ is sufficiently large, then $E(n)=\Phi^{A \oplus
B}(n)=\mathcal C(A \oplus B)(n)$, so $E$ is a coarse description of
 $\mathcal C(A \oplus B)$. Since $E$ is obtained uniformly from $H$,
   we have that $\mathcal{C}(A \oplus B) \leq\sub{uc} G$.
\end{proof}

\begin{lem}
Let $\mathcal{C}$ be a uniform coarsening and let $Y \leq\sub{T} X$. Then $Y
\leq\sub{uc} \mathcal{C}(X)$.
\end{lem}

\begin{proof}
Let $\Phi$ be a Turing functional such that $\Phi^X = Y$. Let $D$ be a
coarse description of $\mathcal{C}(X)$ and let $A_s$ be as in Definition
\ref{coarseningdefn}. Now define $G(n)$ to be the value of
$\Phi^{A_s}(n)$ for the least pair $\langle s,t \rangle$ such that $s
\geq n$ and $\Phi^{A_s}(n)[t]\converges$.  Then $G =^* Y$, so $G$ is a
coarse description of $Y$.
\end{proof}

\begin{prop}
Let $\mathcal{C}$ be a coarsening. Then $\mathcal{C}$ does not always preserve existing
meets in the nonuniform coarse degrees.  The same holds for the uniform coarse
degrees if $\mathcal{C}$ is a uniform coarsening.
\end{prop}

\begin{proof}
Let $X,Y$ be relatively $2$-random and $\Delta^0_3$. Then $X$ and $Y$
form a minimal pair in the Turing degrees, while $X$ and $Y$ do not
form a minimal pair in the nonuniform coarse degrees by Theorem
\ref{2randthm} below. Since every coarse description of $\mathcal{C}(X)$
computes $X$ we see that $\mathcal{C}(X) \geq\sub{nc} X$ and $\mathcal{C}(Y) \geq\sub{nc}
Y$. Therefore $\mathcal{C}(X)$ and $\mathcal{C}(Y)$ also do not form a minimal pair in the
nonuniform coarse degrees.

Next, let $\mathcal{C}$ be a uniform coarsening. We have seen above that there
exists some $A \leq\sub{nc} \mathcal{C}(X),\mathcal{C}(Y)$ that is not coarsely
computable. Then $A \leq\sub{T} X,Y$, so $A \leq\sub{uc} \mathcal{C}(X),\mathcal{C}(Y)$ by
the previous lemma. Thus, $\mathcal{C}(X)$ and $\mathcal{C}(Y)$  do not form a minimal
pair in the uniform coarse degrees.
\end{proof}

\section{Randomness, $K$-triviality, and robust information coding}
\label{Ktriv}

It is reasonable to expect that the embeddings induced by $\mathcal{E}$
(or equivalently, by any uniform coarsening) are not
surjective. Indeed, if $ \mathcal{E}(A) \leq\sub{uc} X$ then the information
represented by $A$ is coded into $X$ in a fairly redundant way. If $A$
is noncomputable, it should follow that $X$ cannot be random. As we
will see, we can make this intuition precise.

\begin{defn}
Let $X^{\mathfrak c}$ be the set of all $A$ such that $A$ is
computable from every coarse description of $X$.
\end{defn}

We will show that if $X$ is weakly $2$-random then $X^{\mathfrak c} =
{\bf 0}$, and hence $ \mathcal{E}(A) \nleq\sub{nc} X$ for all noncomputable $A$
(since every coarse description of $\mathcal{E}(A)$ computes $A$). Since no
$1$-random set can be coarsely computable, it will follow that $X
\not\equiv\sub{nc} \mathcal{E}(B)$ and $X \not\equiv\sub{uc} \mathcal{E}(B)$ for all
$B$. We will first prove the following theorem. Let $\mathcal K$ be
the class of $K$-trivial sets. (See \cite{DH} or \cite{N} for more on
$K$-triviality.)

\begin{thm}
\label{1randandK}
If $X$ is $1$-random then $X^{\mathfrak c} \subseteq \mathcal K$.
\end{thm}

By Downey, Nies, Weber, and Yu \cite{DNWY}, if $X$ is weakly
$2$-random then it cannot compute any noncomputable $\Delta^0_2$
sets. Since $\mathcal K \subset \Delta^0_2$, our desired result
follows from Theorem \ref{1randandK}.

\begin{cor}
\label{weak2randcor}
If $X$ is weakly $2$-random then $X^{\mathfrak c} = {\bf 0}$, and
hence $ \mathcal{E}(A) \nleq\sub{nc} X$ for all noncomputable $A$. In
particular, in both the uniform and nonuniform coarse degrees, the
degree of $X$ is not in the image of the embedding induced by
$\mathcal{E}$.
\end{cor}

To prove Theorem \ref{1randandK}, we use the fact, established by
Hirschfeldt, Nies, and Stephan \cite{HNS}, that $A$ is $K$-trivial
if and only if $A$ is a base for $1$-randomness, that is, $A$ is  computable in a set that
is $1$-random relative to $A$. The basic idea is to show that if $X$
is $1$-random and $A \in X^{\mathfrak c}$, then for each $k > 1$ there is a way to
partition $X$ into $k$ many ``slices'' $X_0,\ldots,X_{k-1}$ such that for each
$i<k$, we have $A \leq\sub{T} X_0 \oplus \cdots \oplus X_{i-1} \oplus X_{i+1}
\oplus \cdots \oplus X_{k-1}$ (where the right hand side of this
inequality denotes $X_1 \oplus \cdots \oplus X_{k-1}$ when $i=0$ and
$X_0 \oplus \cdots \oplus X_{k-2}$
when $i = k-1$). It will then follow by van Lambalgen's
Theorem (which will be discussed below) that each $X_i$ is $1$-random
relative to $X_0 \oplus \cdots \oplus X_{i-1} \oplus X_{i+1} \oplus \cdots
\oplus X_{k-1} \oplus A$, and hence, again by van Lambalgen's Theorem,
that $X$ is $1$-random relative to $A$. Since $A \in X^{\mathfrak c}$
implies that $A \leq\sub{T} X$, we will conclude that $A$ is a base
for $1$-randomness, and hence is $K$-trivial. We begin with some notation
for certain partitions of $X$.

\begin{defn}
\label{partdef}
Let $X \subseteq \omega$. For an infinite subset $Z=\{z_0<z_1<\cdots\}$ of $\omega$, 
let $X \uhr Z = \{n : z_n \in X\}$. For $k>1$ and $i<k$, define $$X^k_i = X
\uhr \{n : n \equiv i \bmod k\}\quad \textrm{and}\quad X^k_{\neq i} =
X \uhr \{n : n \not\equiv i \bmod k\}.$$ Note that $X^k_{\neq i}
\equiv\sub{T} X \setminus \{n : n \equiv i \bmod k\}$ and
$\overline{\rho}(X \triangle (X \setminus \{n : n \equiv i \bmod k\}))
\leq \frac{1}{k}$.
\end{defn}

Van Lambalgen's Theorem \cite{vL} states that $Y \oplus Z$ is
$1$-random if and only if $Y$ and $Z$ are relatively $1$-random. The
proof of this theorem shows, more generally, that if $Z$ is
computable, infinite, and coinfinite, then $X$ is $1$-random if and
only if $X \uhr Z$ and $X \uhr \overline{Z}$ are relatively
$1$-random. Relativizing this fact and applying induction, we get the
following version of van Lambalgen's Theorem.

\begin{thm}[van Lambalgen \cite{vL}]
\label{vlthm}
The following are equivalent for all sets $X$ and $A$, and all $k>1$.
\begin{enumerate}[\rm 1.]

\item $X$ is $1$-random relative to $A$.

\item For each $i<k$, the set $X^k_i$ is $1$-random relative to
$X^k_{\neq i} \oplus A$.

\end{enumerate}
\end{thm}

The last ingredient we need for the proof of Theorem \ref{1randandK}
is a kind of compactness principle, which will also be used to yield
further results in the next section, and is of independent interest
given its connection with the following concept defined in \cite{HJMS}.

\begin{defn}
\label{bounddef}
Let $r \in [0,1]$. A set $X$ is \emph{coarsely computable at density
$r$} if there is a computable set $C$ such that $\overline{\rho}(X
\triangle C) \leq $1$-r$. The \emph{coarse computability
bound} of $X$ is $$\gamma(X) = \sup\{r : X \textrm{ is coarsely
computable at density } r\}.$$
\end{defn}

As noted in \cite{HJMS}, there are sets $X$ such that $\gamma(X)=1$
but $X$ is not coarsely computable. In other words, there is no principle of
``compactness of computable coarse descriptions''. (Although
 Miller (see \cite[Theorem 5.8]{HJMS}) showed that  one can in fact
recover such a principle by adding a further effectivity condition to
the requirement that $\gamma(X)=1$.) The following theorem shows that
if we replace ``computable'' by ``cone-avoiding'', the situation is
different.

\begin{thm}
\label{compthm}
Let $A$ and $X$ be arbitrary sets. Suppose that for each $\epsilon >
0$ there is a set $D_\epsilon$ such that $\overline{\rho}(X \triangle
D_\epsilon) \leq \epsilon$ and $A \nleq\sub{T} D_\epsilon$. Then
there is a coarse description $D$ of $X$ such that $A \nleq\sub{T} D$.
\end{thm}

\begin{proof}
The basic idea is that, given a Turing functional $\Phi$ and a string
$\sigma$ that is ``close to'' $X$, we can extend $\sigma$ to a string
$\tau$ that is ``close to" $X$ such that $\Phi^D \neq A$ for all $D$
extending $\tau$ that are ``close to'' $X$.  We can take $\tau$ to be
any string ``close to" $X$ such that, for some $n$, either
$\Phi^\tau(n) \converges \neq A(n)$ or $\Phi^\gamma(n) \diverges$ for
all $\gamma$ extending $\tau$ that are ``close to" $X$.  If no such
$\tau$ exists, we can obtain a contradiction by arguing that $A
\leq\sub{T} D_\epsilon$ for sufficiently small $\epsilon$, since with
an oracle for $D_\epsilon$ we have access to many strings that are
``close to" $D_\epsilon$ and hence to $X$, by the triangle inequality
for Hamming distance.  In the above discussion the meaning of ``close
to" is different in different contexts, but the precise version will
be given below.  Further, as the construction proceeds, the meaning of
``close to" becomes so stringent that we guarantee that $\rho(X
\triangle D) = 0$.  We now specify the formal details.

We obtain $D$ as $\bigcup_e \sigma_e$, where $\sigma_e \in
2^{<\omega}$ and $\sigma_0 \subsetneq \sigma_1 \subsetneq \cdots$. In
order to ensure that $\rho(X \triangle D) = 0$, we require that for
all $e$ and all $m$ in the interval $[|\sigma_e|, |\sigma_{e+1}|]$,
either $D$ and $X$ agree on the interval $[|\sigma_e|, m)$ or $\rho_m
(X \triangle D) \leq 2^{-|\sigma_e|}$, with the latter true for $m =
|\sigma_{e+1}|$. This condition implies that $\rho_m (X \triangle D)
\leq 2^{-|\sigma_e|}$ for all $m \in [|\sigma_{e+1}|,
|\sigma_{e+2}|]$, and hence that $\rho(X \triangle D) = 0$.

Let $\sigma$ and $\tau$ be strings and let $\epsilon$ be a positive
real number.  Call $\tau$ an $\epsilon$-\emph{good} extension of
$\sigma$ if $\tau$ properly extends $\sigma$ and for all $m \in
[|\sigma|, |\tau|]$, either $X$ and $\tau$ agree on $[|\sigma|,m)$ or
$\rho_m (\tau \triangle X) \leq \epsilon$, with the latter true for
$m = |\tau|$.  In line with the previous paragraph, we require that
$\sigma_{e+1}$ be a $2^{-|\sigma_e|}$-good extension of $\sigma_e$
for all $e$.

At stage $0$, let $\sigma_0$ be the empty string. At stage $e+1$, we
are given $\sigma_e$ and choose $\sigma_{e+1}$ as follows so as to
force that $ A \neq \Phi_e^D$. Let $\epsilon = 2^{-|\sigma_e|}$.

\emph{Case 1}.  There is a number $n$ and a string $\tau$ that is an
$\epsilon$-good extension of $\sigma_e$ such that $\Phi^{\tau}_e(n)
\converges \neq A(n)$.  Let $\sigma_{e+1}$ be such a $\tau$.

\emph{Case 2}. Case 1 does not hold and there is a number $n$ and a
string $\beta$ that is an $\epsilon$-good extension of $\sigma_e$ such
that $|\beta| \geq |\sigma_e| + 2$ and $\Phi_e^\tau(n) \diverges$ for
all $\frac{\epsilon}{4}$-good extensions $\tau$ of $\beta$.  Let
$\sigma_{e+1}$ be such a $\beta$.

We claim that either Case 1 or Case 2 applies.  Suppose not.  Let
$D_{\frac{\epsilon}{5}}$ be as in the hypothesis of the lemma, so that
$\overline{\rho}(X \triangle D_{\frac{\epsilon}{5}}) \leq
\frac{\epsilon}{5}$ and $A \nleq\sub{T} D_{\frac{\epsilon}{5}}$. Let
$c \geq |\sigma_e| + 2$ be sufficiently large so that $\rho_m (X
\triangle D_{\frac{\epsilon}{5}}) \leq \frac{\epsilon}{4}$ for all $m
\geq c$ and $\sigma_e$ has an $\frac{\epsilon}{4}$-good extension
$\beta$ of length $c$. Note that the string obtained from $\sigma_e$
by appending a sufficiently long segment of $X$ starting with
$X(|\sigma_e|)$ is an $\frac{\epsilon}{4}$-good extension of
$\sigma_e$, so such a $\beta$ exists, and we assume it is obtained
in this manner.

We now obtain a contradiction by showing that $A \leq\sub{T}
D_{\frac{\epsilon}{5}}$.  To calculate $A(n)$ search for a string
$\gamma$ extending $\beta$ such that $\Phi_e^\gamma (n) \converges$,
say with use $u$, and $\rho_m (D_{\frac{\epsilon}{5}} \triangle
\gamma) \leq \frac{\epsilon}{2}$ for all $m \in [c,u)$.  We first
check that such a string $\gamma$ exists.  Since Case 2 does not
hold, there is a string $\tau$ that is an $\frac{\epsilon}{4}$-good
extension of $\beta$ such that $\Phi_e^\tau(n) \converges$.  We
claim that $\tau$ meets the criteria to serve as $\gamma$.  We need
only check that $\rho_m (D_{\frac{\epsilon}{5}} \triangle \tau) \leq
\frac{\epsilon}{2}$ for all $m \in [c,u)$.  Fix $m \in [c,u)$.  Then
$$\rho_m (D_{\frac{\epsilon}{5}} \triangle \tau) \leq \rho_m
(D_{\frac{\epsilon}{5}} \triangle X) + \rho_m (X \triangle \tau)
\leq \frac{\epsilon}{4} + \frac{\epsilon}{4} =
\frac{\epsilon}{2}.$$

Next we claim that $\gamma$ is an $\epsilon$-good extension of
$\sigma_e$.  The string $\gamma$ extends $\sigma_e$ since it extends
$\beta$, and $\beta$ extends $\sigma_e$. Let $m \in [|\sigma_e,
|\gamma|]$ be given.  If $m < c$, then $\gamma$ and $X$ agree on the
interval $[|\sigma_e|, m)$ because $\beta$ and $X$ agree on this
interval and $\gamma$ extends $\beta$. Now suppose that $m \geq c$.
Then
$$\rho_m(\gamma \triangle X) \leq \rho_m(\gamma \triangle
D_{\frac{\epsilon}{5}}) + \rho_m(D_{\frac{\epsilon}{5}} \triangle X)
\leq \frac{\epsilon}{2} + \frac{\epsilon}{4} < \epsilon.$$ Since
$\gamma$ is an $\epsilon$-good extension of $\sigma_e$ for which
$\Phi_e^\gamma(n) \converges$, and Case 1 does not hold, we conclude
that $\Phi_e^\gamma(n) = A(n)$. The search for $\gamma$ can be
carried out computably in $D_{\frac{\epsilon}{5}}$, so we conclude
that $A \leq\sub{T} D_{\frac{\epsilon}{5}}$, contradicting our choice
of $D_{\frac{\epsilon}{5}}$. (Although $\beta$ cannot be computed from
$D_{\frac{\epsilon}{5}}$, we may use it in our computation of $A(n)$
since it is a fixed string which does not depend on $n$.) This
contradiction shows that Case 1 or Case 2 must apply.

Let $D = \bigcup_n \sigma_n$.  Then $\rho(D \triangle X) = 0$, and $A
\nleq\sub{T} D$ since Case 1 or Case 2 applies at every stage.
\end{proof}

\begin{proof}[Proof of Theorem \ref{1randandK}]
Let $A \in X^{\mathfrak c}$. By Theorem \ref{compthm}, there is an
$\epsilon > 0$ such that $A \leq\sub{T} D_\epsilon$ whenever
$\overline{\rho}(X \triangle D_\epsilon) \leq \epsilon$. Let $k$ be an
integer such that $k > \frac{1}{\epsilon}$. As noted in
Definition \ref{partdef}, $X^k_{\neq i}$ is Turing equivalent to such
a $D_\epsilon$ for each $i < k$, so we have $A \leq\sub{T} X^k_{\neq i}$
for all $i<k$. By the unrelativized form of Theorem \ref{vlthm}, each
$X^k_i$ is $1$-random relative to $X^k_{\neq i}$, and hence relative
to $X^k_{\neq i} \oplus A \equiv\sub{T} X^k_{\neq i}$. Again by
Theorem \ref{vlthm}, $X$ is $1$-random relative to $A$. But $A
\leq\sub{T} X$, so $A$ is a base for $1$-randomness, and hence is
$K$-trivial.
\end{proof}

Weak $2$-randomness is exactly the level of randomness necessary to
obtain Corollary \ref{weak2randcor} directly from Theorem
\ref{1randandK}, because, as shown in \cite{DNWY}, if a $1$-random set
is not weakly $2$-random, then it computes a noncomputable
c.e.\ set. The corollary itself does hold of some $1$-random sets that
are not weakly $2$-random, because if it holds of $X$ then it also
holds of any $Y$ such that $\rho(Y \triangle X)=0$. (For example, let
$X$ be $2$-random and let $Y$ be obtained from $X$ by letting
$Y(2^n)=\Omega(n)$ (where $\Omega$ is Chaitin's halting probability)
for all $n$ and letting $Y(k)=X(k)$ for all other $k$. By van
Lambalgen's Theorem, $Y$ is $1$-random, but it computes $\Omega$, and
hence is not weakly $2$-random.)

Nevertheless, Corollary \ref{weak2randcor} does not hold of all
$1$-random sets, as we now show.

\begin{defn}
Let $W_0,W_1,\ldots$ be an effective listing of the c.e.\ sets. A set
$A$ is \emph{promptly simple} if it is c.e.\ and coinfinite, and there
exist  a computable function $f$ and a computable enumeration
$A[0],A[1],\ldots$ of $A$ such that for each $e$, if $W_e$ is infinite
then there are $n$ and $s$ for which $n \in W_e[s] \setminus W_e[s-1]$
and $n \in A[f(s)]$. Note that every promptly simple set is
noncomputable.
\end{defn}

We will show that if $X \leq\sub{T} \emptyset'$ is $1$-random then
$X^{\mathfrak c}$ contains a promptly simple set, and there is a
promptly simple set $A$ such that $ \mathcal{E}(A) \leq\sub{nc} X$. (We do not
know whether we can improve the last statement to $ \mathcal{E}(A) \leq\sub{uc}
X$.) In fact, we will obtain a considerably stronger result by first
proving a generalization of the fact, due to Hirschfeldt and Miller
(see \cite[Theorem 7.2.11]{DH}), that if $\mathcal T$ is a $\Sigma^0_3$
class of measure $0$, then there is a noncomputable c.e.\ set that is
computable from each $1$-random element of $\mathcal T$.

For a binary relation $P(Y,Z)$ between elements of $2^\omega$, let
$P(Y)=\{Z : P(Y,Z)\}$.

\begin{thm}
\label{sigma3thm}
Let $\mathcal S_0,\mathcal S_1,\ldots$ be uniformly $\Pi^0_2$ classes
of measure $0$, and let $P_0(Y,Z),P_1(Y,Z),\ldots$ be uniformly
$\Pi^0_1$ relations. Let $\mathcal D$ be the class of all $Y$ for
which there are numbers $k,m$ and a $1$-random set $Z$ such that $Z
\in P_k(Y) \subseteq \mathcal S_m$. Then there is a promptly simple
set $A$ such that $A \leq\sub{T} Y$ for every $Y \in \mathcal D$.
\end{thm}

\begin{proof}
Let $(\mathcal V^m_n)_{m,n \in \omega}$ be uniformly $\Sigma^0_1$
classes such that $\mathcal S_m = \bigcap_n \mathcal V^m_n$. We may
assume that $\mathcal V^m_0 \supseteq \mathcal V^m_1 \supseteq \cdots$
for all $m$. For each $m$, we have $\mu(\bigcap_n \mathcal V^n_m) =
\mu(\mathcal S_m)=0$, so $\lim_n \mu(\mathcal V^m_n)=0$ for each
$m$. Let $\Theta$ be a computable relation such that $P_k(Y,Z) \equiv
\forall l\, \Theta(k,Y \uhr l,Z \uhr l)$.

Define $A$ as follows. At each stage $s$, if there is an $e<s$ such
that no numbers have entered $A$ for the sake of $e$ yet, and an
$n>2e$ such that $n \in W_e[s] \setminus W_e[s-1]$ and $\mu(\mathcal
V^m_n[s]) \leq 2^{-e}$ for all $m<e$, then for the least such $e$, put
the least corresponding $n$ into $A$. We say that $n$ enters $A$ for
the sake of $e$.

Clearly, $A$ is c.e.\ and coinfinite, since at most $e$ many numbers
less than $2e$ ever enter $A$. Suppose that $W_e$ is infinite. Let
$t>e$ be a stage such that all numbers that will ever enter $A$ for
the sake of any $i<e$ are in $A[t]$. There must be an $s \geq t$ and
an $n>2e$ such that $n \in W_e[s] \setminus W_e[s-1]$ and
$\mu(\mathcal V^m_n[s]) \leq 2^{-e}$ for all $m<e$. Then the least
such $n$ enters $A$ for the sake of $e$ at stage $s$ unless another
number has already entered $A$ for the sake of $e$. It follows that
$A$ is promptly simple.

Now suppose that $Y \in \mathcal D$. Let the numbers $k,m$ and the
$1$-random set $Z$ be such that $Z \in P_k(Y) \subseteq \mathcal S_m$.
Let $B \leq\sub{T} Y$ be defined as follows. Given $n$, let $$\mathcal
D^n_s = \{X : (\forall l \leq s)\, \Theta(k,Y \uhr l,X \uhr l)\}
\setminus \mathcal V^m_n[s].$$ Then $\mathcal D^n_0 \supseteq \mathcal
D^n_1 \supseteq \cdots$. Furthermore, if $X \in \bigcap_s \mathcal
D^n_s$ then $P_k(Y,X)$ and $X \notin \mathcal V^m_n$. Since $P_k(Y)
\subseteq S_m \subseteq \mathcal V^m_n$, it follows that $X \notin
P_k(Y)$, which is a contradiction. Thus $\bigcap_s \mathcal D^n_s =
\emptyset$. Since the $\mathcal D^n_s$ are nested closed sets, it
follows that there is an $s$ such that $\mathcal D^n_s =
\emptyset$. Let $s_n$ be the least such $s$ (which we can find using
$Y$) and let $B(n)=A(n)[s_n]$. Note that $B \subseteq A$.

Let $T=\{\mathcal V^m_n[s] : n \textrm{ enters } A \textrm{ at stage }
s\}$. We can think of $T$ as a uniform singly-indexed sequence of
$\Sigma^0_1$ sets since $m$ is fixed and for each $n$ there is at most
one $s$ such that $\mathcal V^m_n[s] \in T$. For each $e$, there is at
most one $n$ that enters $A$ for the sake of $e$, and the sum of the
measures of the $\mathcal V^m_n[s]$ such that $n$ enters $A$ at stage
$s$ for the sake of some $e>m$ is bounded by $\sum_e 2^{-e}$, which is
finite. Thus $T$ is a Solovay test, and hence $Z$ is in only finitely
many elements of $T$. So for all but finitely many $n$, if $n$ enters
$A$ at stage $s$ then $Z \notin \mathcal V^m_n[s]$. Then $Z \in
\mathcal D^n_s$, so $s_n>s$. Hence, for all such $n$, we have that
$B(n) = A(n)[s_n] = 1$. Thus $B =^* A$, so $A \equiv\sub{T} B
\leq\sub{T} Y$.
\end{proof}

Note that the result of Hirschfeldt and Miller mentioned above follows
from this theorem by starting with a $\Sigma^0_3$ class $\mathcal
S=\bigcap_m \mathcal S_m$ of measure $0$ and letting each $P_k$ be the
identity relation.

\begin{cor}
\label{1randcor}
Let $X \leq\sub{T} \emptyset'$ be $1$-random. There is a promptly
simple set $A$ such that if $\overline{\rho}(D \triangle
X)<\frac{1}{4}$ then $A \leq\sub{T} D$. In particular, $X^{\mathfrak
c}$ contains a promptly simple set, and there is a promptly simple
set $A$ such that $ \mathcal{E}(A) \leq\sub{nc} X$.
\end{cor}

\begin{proof}
Say that sets $Y$ and $Z$ are \emph{$r$-close from $m$ on} if whenever
$m < n$, the Hamming distance between $Y \uhr n$ and $Z \uhr n$
(i.e., the number of bits on which these two strings differ) is at
most $rn$.

Let $\mathcal S_m$ be the class of all $Z$ such that $X$ and $Z$ are
$\frac{1}{2}$-close from $m$ on. Since $X$ is $\Delta^0_2$, the
$\mathcal S_m$ are uniformly $\Pi^0_2$ classes. Furthermore, if $X$
and $Z$ are $\frac{1}{2}$-close from $m$ on for some $m$, then $Z$
cannot be $1$-random relative to $X$ (by the same argument that shows
that if $C$ is $1$-random then there must be infinitely many $n$ such
that $C \uhr n$ has more $1$'s than $0$'s), so $\mu(\mathcal S_m)=0$
for all $m$. Let $P_m(Y,Z)$ hold if and only if $Y$ and $Z$ are
$\frac{1}{4}$-close from $m$ on. The $P_m$ are clearly uniformly
$\Pi^0_1$ relations.

Thus the hypotheses of Theorem \ref{sigma3thm} are satisfied. Let $A$
be as in that theorem. Suppose that $\overline{\rho}(D \triangle
X)<\frac{1}{4}$. Then there is an $m$ such that $D$ and $X$ are
$\frac{1}{4}$-close from $m$ on. If $D$ and $Z$ are
$\frac{1}{4}$-close from $m$ on, then by the triangle inequality for
Hamming distance, $X$ and $Z$ are $\frac{1}{2}$-close from $m$
on. Thus $X \in P_m(D) \subseteq \mathcal S_m$, so $A \leq\sub{T} D$.
\end{proof}

After learning about Corollary \ref{1randcor}, Nies \cite{N2} gave a
different but closely connected proof of this result, which works even
for $X$ of positive effective Hausdorff dimension, as long as we
sufficiently decrease the bound $\frac{1}{4}$. However, even for $X$
of effective Hausdorff dimension $1$ his bound is much worse, namely
$\frac{1}{20}$.

Maass, Shore, and Stob \cite[Corollary 1.6]{MSS} showed that if $A$
and $B$ are promptly simple then there is a promptly simple set $G$ such
that $G \leq\sub{T} A$ and $G \leq\sub{T} B$. Thus we have the
following extension of Ku{\v c}era's result \cite{K} that two
$\Delta^0_2$ $1$-random sets cannot form a minimal pair, which will also
be useful below.

\begin{cor}
\label{quartercor}
Let $X_0,X_1 \leq\sub{T} \emptyset'$ be $1$-random. There is a
promptly simple set $A$ such that if $\overline{\rho}(D \triangle
X_i)<\frac{1}{4}$ for some $i \in \{0,1\}$ then $A \leq\sub{T} D$.
\end{cor}

It is easy to adapt the proof of Corollary \ref{1randcor} to give a
direct proof of Corollary \ref{quartercor}, and indeed of the fact
that for any uniformly $\emptyset'$-computable family $X_0,X_1,\ldots$
of $1$-random sets, there is a promptly simple set $A$ such that if
$\overline{\rho}(D \triangle X_i)<\frac{1}{4}$ for some $i$ then $A
\leq\sub{T} D$. (We let $\mathcal S_{\langle i,m \rangle}$ be the
class of all $Z$ such that $X_i$ and $Z$ are $\frac{1}{2}$-close from
$m$ on, and the rest of the proof is essentially as before.)

Given the many (and often surprising) characterizations of
$K$-triviality, it is natural to ask whether there is a converse to
Theorem \ref{1randandK} stating that if $A$ is $K$-trivial then $A \in
X^{\mathfrak c}$ for some $1$-random $X$. We now show that is not the
case, using a recent result of Bienvenu, Greenberg, Ku{\v c}era, Nies,
and Turetsky \cite{BGKNT}. There are many notions of randomness tests
in the theory of algorithmic randomness. Some, like Martin-L\"of
tests, correspond to significant levels of algorithmic randomness,
while other, less obviously natural ones have nevertheless become
important tools in the development of this theory. Balanced tests
belong to the latter class.

\begin{defn}
Let $\mathcal W_0,\mathcal W_1,\ldots \subseteq 2^\omega$ be an
effective list of all $\Sigma^0_1$ classes. A \emph{balanced test} is
a sequence $(\mathcal U_n)_{n \in \omega}$ of $\Sigma^0_1$ classes such
that there is a computable binary function $f$ with the following
properties.
\begin{enumerate}[\rm 1.]

\item $|\{s : f(n,s+1) \neq f(n,s)\}| \leq O(2^n)$,

\item $\forall n\; \mathcal U_n = \mathcal W_{\lim_s f(n,s)}$, and

\item $\forall n\; \forall s\; \mu(\mathcal W_{f(n,s)}) \leq 2^{-n}$.

\end{enumerate}
\end{defn}

For $\sigma \in 2^{<\omega}$ and $X \in 2^\omega$, we write $\sigma X$
for the element of $2^\omega$ obtained by concatenating $\sigma$ and
$X$.

\begin{thm}[Bienvenu, Greenberg, Ku{\v c}era, Nies, and Turetsky
\cite{BGKNT}]
\label{oberthm}
There are a $K$-trivial set $A$ and a balanced test $(\mathcal U_n)_{n
\in \omega}$ such that if $A \leq\sub{T} X$ then there is a string
$\sigma$ with $\sigma X \in \bigcap_n \mathcal U_n$.
\end{thm}

We will also use the following measure-theoretic fact.

\begin{thm}[Loomis and Whitney \cite{LW}]
\label{lwthm}
Let $\mathcal S \subseteq 2^\omega$ be open, and let $k \in
\omega$. For $i<k$, let $\pi_i(\mathcal S)=\{Y^k_{\neq i} : Y \in
\mathcal S\}$. Then $\mu(\mathcal S)^{k-1} \leq \mu(\pi_0(\mathcal
S))\cdots\mu(\pi_{k-1}(\mathcal S))$.
\end{thm}

Our result will follow from the following lemma.

\begin{lem}
\label{balancedlem}
Let $X$ be $1$-random, let $k>1$, and let $(\mathcal U_n)_{n \in
  \omega}$ be a balanced test. There is an $i<k$ such that $X^k_{\neq
  i} \notin \bigcap_n \mathcal U_n$.
\end{lem}

\begin{proof}
Assume for a contradiction that $X^k_{\neq i} \in \bigcap_n \mathcal
U_n$ for all $i<k$. Let
$$
\mathcal S_{n,s} = \{Y : \forall i<k\;
(Y^k_{\neq i} \in \mathcal U_n[s])\}
$$
and let $\mathcal S_n=\bigcup_s \mathcal S_{n,s}$. By Theorem
\ref{lwthm}, $\mu(\mathcal S_{n,s})^{k-1} \leq \mu(\mathcal
U_n[s])^k$, so $\mu(\mathcal S_n) \leq O(2^n)2^{-\frac{nk}{k-1}} =
O(2^{-\frac{n}{k-1}})$, and hence $\sum_n \mu(\mathcal
S_n)<\infty$. Thus $\{\mathcal S_n : n \in \omega\}$ is a Solovay
test. However, $X \in \bigcap_n \mathcal S_n$, so we have a
contradiction.
\end{proof}

\begin{thm}
There is a $K$-trivial set $A$ such that $A \notin X^{\mathfrak c}$
for all $1$-random $X$.
\end{thm}

\begin{proof}
Let $A$ and $(\mathcal U_n)_{n \in \omega}$ be as in Theorem
\ref{oberthm}. Let $X$ be $1$-random. By Theorem \ref{compthm}, it is
enough to fix $k>1$ and show that there is an $i<k$ such that $A
\nleq\sub{T} X^k_{\neq i}$. Assume for a contradiction that $A
\leq\sub{T} X^k_{\neq i}$ for all $i<k$. Then there are
$\sigma_0,\ldots,\sigma_{k-1}$ such that $\sigma_iX^k_{\neq i} \in
\bigcap_n \mathcal U_n$ for all $i<k$. Let $m=\max_{i<k} |\sigma_i|$
and let $\mathcal V_n=\{Y : \exists i<k\; (\sigma_iY \in \mathcal
U_{n+k+m})\}$. It is easy to check that $(\mathcal V_n)_{n \in \omega}$
is a balanced test, and $X^k_{\neq i} \in \bigcap_n \mathcal V_n$ for
all $i<k$, which contradicts Lemma \ref{balancedlem}.
\end{proof}

\section{Further applications of cone-avoiding compactness} 
\label{further}

We can use Theorem \ref{compthm} to give an analog to Corollary
\ref{weak2randcor} for effective genericity. In this case,
$1$-genericity is sufficient, as it is straightforward to show that if
$X$ is $1$-generic relative to $A$ and $A$ is noncomputable, then $A
\nleq\sub{T} X$ (i.e., unlike the case for $1$-randomness, there are
no noncomputable bases for $1$-genericity), and that no $1$-generic
set can be coarsely computable. The other ingredient we need to
replicate the argument we gave in the case of effective randomness is
a version of van Lambalgen's Theorem for $1$-genericity. This result
was established by Yu \cite[Proposition 2.2]{Y}. Relativizing his
theorem and applying induction as in the case of Theorem \ref{vlthm},
we obtain the following fact.

\begin{thm}[Yu \cite{Y}]
\label{ythm}
The following are equivalent for all sets $X$ and $A$, and all $k>1$.
\begin{enumerate}[\rm 1.]

\item $X$ is $1$-generic relative to $A$.

\item For each $i<k$, the set $X^k_i$ is $1$-generic relative to
$X^k_{\neq i} \oplus A$.

\end{enumerate}
\end{thm}

Now we can establish the following analog to Corollary
\ref{weak2randcor}.

\begin{thm}
\label{1genthm}
If $X$ is $1$-generic then $X^{\mathfrak c} = {\bf 0}$, and hence
$ \mathcal{E}(A) \nleq\sub{nc} X$ for all noncomputable $A$. In
particular, in both the uniform and nonuniform coarse degrees, the
degree of $X$ is not in the image of the embedding
induced by $\mathcal{E}$.
\end{thm}

\begin{proof}
Let $A \in X^{\mathfrak c}$. As in the proof of Theorem
\ref{1randandK}, there is a $k$ such that $A \leq\sub{T} X^k_{\neq i}$
for all $i<k$. By the unrelativized form of Theorem \ref{ythm}, each
$X^k_i$ is $1$-generic relative to $X^k_{\neq i}$, and hence relative
to $X^k_{\neq i} \oplus A \equiv\sub{T} X^k_{\neq i}$. Again by
Theorem \ref{ythm}, $X$ is $1$-generic relative to $A$. But $A
\leq\sub{T} X$, so $A$ is computable.
\end{proof}

Igusa (personal communication) has also found the following
application of Theorem \ref{compthm}. We say that $X$ is
\emph{generically computable} if there is a partial computable
function $\varphi$ such that $\varphi(n) = X(n)$ for all $n$ in the
domain of $\varphi$, and the domain of $\varphi$ has density
$1$. Jockusch and Schupp \cite[Theorem 2.26]{JS} showed that there are
generically computable sets that are not coarsely computable, but by
Lemma 1.7 in \cite{HJMS}, if $X$ is generically computable then
$\gamma(X)=1$, where $\gamma$ is the coarse computability bound from
Definition \ref{bounddef}.

\begin{thm}[Igusa, personal communication]
If $\gamma(X)=1$ then $X^{\mathfrak c} = {\bf 0}$, and hence $\mathcal{E}(A)
\nleq\sub{nc} X$ for all noncomputable $A$. Thus, if $\gamma(X) = 1$
and $X$ is not coarsely computable then in both the uniform and
nonuniform coarse degrees, the degree
of $X$ is not in the image of the embedding induced by $\mathcal{E}$. In
particular, the above holds when $X$ is generically computable but not
coarsely computable.
\end{thm}

\begin{proof}
Suppose that $\gamma(X)=1$ and $A$ is not computable. If $\epsilon>0$
then there is a computable set $C$ such that $\overline{\rho}(X
\triangle C) < \epsilon$. Since $C$ is computable, $A \nleq\sub{T}
C$. By Theorem \ref{compthm}, $A \notin X^{\mathfrak c}$.
\end{proof}

\section{Minimal pairs in the uniform and nonuniform\\ coarse degrees}
\label{minpairs}

For any degree structure that acts as a measure of information
content, it is reasonable to expect that if two sets are sufficiently
random relative to each other, then their degrees form a minimal
pair. For the Turing degrees, it is not difficult to show that if $Y$
is not computable and $X$ is weakly $2$-random relative to $Y$, then
the degrees of $X$ and $Y$ form a minimal pair. On the other hand,
Ku{\v c}era \cite{K} showed that if $X,Y \leq\sub{T} \emptyset'$ are
both $1$-random, then there is a noncomputable set $A \leq\sub{T}
X,Y$, so there are relatively $1$-random sets whose degrees do not
form a minimal pair. As we will see, the situation for the
nonuniform coarse degrees is similar, but ``one jump up''.

For an interval $I$, let $\rho_I(X)=\frac{|X \cap I|}{|I|}$.

\begin{lem}
\label{jlem}
Let $J_k=[2^{k}-1,2^{k+1}-1)$. Then $\rho(X)=0$ if and only if $\lim_k
\rho_{J_k}(X)=0$.
\end{lem}

\begin{proof}
First suppose that $\limsup_k \rho_{J_k}(X) > 0$. Since $|J_k| = 2^k$,
we have $\overline{\rho}(X) \geq \limsup_k \rho_{2^{k+1}-1}(X) \geq
\limsup_k \frac{\rho_{J_k}(X)}{2} > 0$.

Now suppose that $\limsup_k \rho_{J_k}(X) = 0$. Fix
$\epsilon>0$. If $m$ is sufficiently large, $k \geq m$, and $n \in
J_k$, then $$|X \cap [0,n)| \leq |X \cap [0,2^{k+1}-1)| \leq
\sum_{i=0}^{m-1} |J_i| + \sum_{i=m}^k \frac{\epsilon}{2}
|J_i|.$$ If $k$ is sufficiently large then this sum is less than
$\epsilon (2^k-1)$, whence $\rho_n(X) < \frac{\epsilon
(2^k-1)}{n} \leq \frac{\epsilon n}{n}=\epsilon$. Thus
$\limsup_n \rho_n(X) \leq \epsilon$. Since $\epsilon$ is
arbitrary, $\limsup_n \rho_n(X) = 0$.
\end{proof}

\begin{thm}
If $A$ is not coarsely computable and $X$ is weakly $3$-random
relative to $A$, then there is no $X$-computable coarse description of
$A$.  In particular,  $A \nleq\sub{nc} X$.
\end{thm}

\begin{proof}
Suppose that $\Phi^X$ is a coarse description of $A$ and let $$\mathcal
P=\{Y : \Phi^Y \textrm{ is a coarse description of } A\}.$$ Then $Y \in
\mathcal P$ if and only if
\begin{enumerate}[\rm 1.]

\item $\Phi^Y$ is total, which is a $\Pi^0_2$ property, and

\item for each $k$ there is an $m$ such that, for all $n>m$, we have
$\rho_n(\Phi^Y \triangle A)<2^{-k}$, which is a $\Pi^{0,A}_3$ property.

\end{enumerate}
Thus $\mathcal P$ is a $\Pi^{0,A}_3$ class, so it suffices to show
that if $A$ is not coarsely computable then $\mu(\mathcal P)=0$.

We prove the contrapositive. Suppose that $\mu(\mathcal P)>0$. Then,
by the Lebesgue Density Theorem, there is a $\sigma$ such that
$\mu(\mathcal P \cap \open{\sigma})>\frac{3}{4}2^{-|\sigma|}$. It is
now easy to define a Turing functional $\Psi$ such that the measure of
the class of $Y$ for which $\Psi^Y$ is a coarse description of $A$ is
greater than $\frac{3}{4}$. Define a computable set $D$ as
follows. Let $J_k=[2^{k}-1,2^{k+1}-1)$.  For each $k$, wait until we
find a finite set of strings $S_k$ such that
$\mu(\open{S_k})>\frac{3}{4}$ and $\Psi^\sigma$ converges on all of
$J_k$ for each $\sigma \in S_k$ (which must happen, by our choice of
$\Psi$). Let $n_k$ be largest such that there is a set $R_k
\subseteq S_k$ with $\mu(\open{R_k})>\frac{1}{2}$ and
$\rho_{J_k}(\Psi^\sigma \triangle \Psi^\tau) \leq 2^{-n_k}$ for all
$\sigma,\tau \in R_k$. Let $\sigma \in R_k$ and define $D \uhr J_k =
\Psi^\sigma \uhr J_k$.

We claim that $D$ is a coarse description of $A$. By Lemma \ref{jlem},
it is enough to show that $\lim_k \rho_{J_k}(D \triangle A)=0$. Fix
$n$. Let $\mathcal B_k$ be the class of all $Y$ such that $\Psi^Y$
converges on all of $J_k$ and $\rho_{J_k}(\Psi^Y \triangle A) \leq
2^{-n}$. If $\Psi^Y$ is a coarse description of $A$ then, again by
Lemma \ref{jlem}, $\rho_{J_k}(\Psi^Y \triangle A) \leq 2^{-n}$ for all
sufficiently large $k$, so there is an $m$ such that $\mu(\mathcal
B_k)>\frac{3}{4}$ for each $k>m$, and hence $\mu(\mathcal B_k \cap
\open{S_k})>\frac{1}{2}$ for each $k>m$. Let $T_k = \{\sigma \in S_k :
\rho_{J_k}(\Psi^\sigma \triangle A) \leq 2^{-n}\}$. Then $\open{T_k} =
\mathcal{B}_k \cap \open{S_k}$, so $\mu(\open{T_k}) > \frac{1}{2}$ for
each $k > m$. Furthermore, by the triangle inequality for Hamming
distance, $\rho_{J_k}(\Psi^\sigma \triangle \Psi^\tau) \leq
2^{-(n-1)}$ for all $\sigma, \tau \in T_k$. It follows that, for each
$k>m$, we have $n_k \geq n-1$, and at least one element $Y$ of
$\mathcal B_k$ is in $\open{R_k}$ (where $R_k$ is as in the definition
of $D$), which implies that $$\rho_{J_k}(D \triangle A) \leq
\rho_{J_k}(D \triangle \Psi^Y) + \rho_{J_k}(\Psi^Y \triangle A) \leq
2^{-n_k}+2^{-n} < 2^{-n+2}.$$ Since $n$ is arbitrary, $\lim_k
\rho_{J_k}(D \triangle A)=0$.
\end{proof}

\begin{cor}
If $Y$ is not coarsely computable and $X$ is weakly $3$-random
relative to $Y$, then the nonuniform coarse degrees of $X$ and $Y$
form a minimal pair, and hence so do their uniform coarse degrees.
\end{cor}

\begin{proof}
Let $A \leq\sub{nc} X,Y$. Then $Y$ computes a coarse description $D$
of $A$. We have $D \leq\sub{nc} X$, and $X$ is weakly $3$-random
relative to $D$, so by the theorem, $D$ is coarsely computable, and
hence so is $A$.
\end{proof}

For the nonuniform coarse degrees at least, this corollary does not
hold of $2$-randomness in place of weak $3$-randomness. To establish
this fact, we use the following complementary results. The first was
proved by Downey, Jockusch, and Schupp \cite[Corollary 3.16]{DJS} in
unrelativized form, but it is easy to check that their proof
relativizes.

\begin{thm}[Downey, Jockusch, and Schupp \cite{DJS}]
\label{DJSthm}
If $A$ is c.e., $\rho(A)$ is defined, and $A' \leq\sub{T} D'$, then $D$
computes a coarse description of $A$.
\end{thm}

\begin{thm}[Hirschfeldt, Jockusch, McNicholl, and Schupp \cite{HJMS}]
\label{HJMSthm}
Every nonlow c.e.\ degree contains a c.e.\ set $A$ such that
$\rho(A)=\frac{1}{2}$ and $A$ is not coarsely computable.
\end{thm}

\begin{thm}
\label{2randthm}
Let $X,Y \leq\sub{T} \emptyset''$ (which is equivalent to $\mathcal{E}(X),\mathcal{E}(Y)
\leq\sub{nc} \mathcal{E}(\emptyset'')$). If $X$ and $Y$ are both $2$-random,
then there is an $A \leq\sub{nc} X,Y$ such that $A$ is not coarsely
computable. In particular, there is a pair of relatively $2$-random
sets whose nonuniform coarse degrees do not form a minimal pair.
\end{thm}

\begin{proof}
Since $X$ and $Y$ are both $1$-random relative to $\emptyset'$, by the
relativized form of Corollary \ref{quartercor} there is an
$\emptyset'$-c.e.\ set $J >\sub{T} \emptyset'$ such that for every
coarse description $D$ of either $X$ or $Y$, we have that $D \oplus
\emptyset'$ computes $J$, and hence so does $D'$. By the Sacks Jump
Inversion Theorem \cite{Sa}, there is a c.e.\ set $B$ such that $B'
\equiv\sub{T} J$. By Theorem \ref{HJMSthm}, there is a c.e.\ set $A
\equiv\sub{T} B$ such that $\rho(A)=\frac{1}{2}$ and $A$ is not
coarsely computable. Let $D$ be a coarse description of either $X$ or
$Y$. Then $D' \geq\sub{T} J \equiv\sub{T} A'$, so by Theorem
\ref{DJSthm}, $D$ computes a coarse description of $A$.
\end{proof}

We do not know whether this theorem holds for uniform coarse reducibility.

\section{Open Questions}
\label{questions}

We finish with a few questions raised by our results.

\begin{oq}
Can the bound $\frac{1}{4}$ in Corollary \ref{1randcor} be increased?
\end{oq}
 
\begin{oq}
Let $X \leq\sub{T} \emptyset'$ be $1$-random. Must there be a
noncomputable (c.e.)\ set $A$ such that $ \mathcal{E}(A) \leq\sub{uc} X$? (Recall
that Corollary \ref{1randcor} gives a positive answer to the
nonuniform analog to this question.) If not, then is there any
$1$-random $X$ for which such an $A$ exists?
\end{oq}

\begin{oq}
Does Theorem \ref{2randthm} hold for uniform coarse reducibility?
\end{oq}

\end{document}